\documentclass[11pt]{article}

\usepackage{amsmath,amsthm,verbatim,amssymb,amsfonts,amscd, graphicx}
\usepackage{graphics}
\usepackage{authblk}
\usepackage{upgreek}
\usepackage{cite}
\theoremstyle{plain}
\newtheorem{theorem}{Theorem}[section]

\newtheorem{corollary}[theorem]{Corollary}

\newtheorem{proposition}[theorem]{Proposition}
\newtheorem{remark}[theorem]{Remark}
\newtheorem*{theorem*}{Theorem}
\numberwithin{equation}{section}
\newcommand{\diff}{\mathop{}\!\mathrm{d}}
\DeclareMathOperator*{\Hess}{Hess}
\DeclareMathOperator*{\tr}{tr}

\DeclareMathOperator*{\esssup}{ess\,sup}
\graphicspath{ {./images/} }

\title{A locally anisotropic regularity criterion for the Navier--Stokes equation in terms of vorticity}
\author[1]{Evan Miller}
\affil[1]{McMaster University, Department of Mathematics and Statistics

millee14@mcmaster.ca}

\begin{document}

\maketitle

\begin{abstract}
In this paper, we will prove a regularity criterion that guarantees solutions of the Navier--Stokes equation must remain smooth so long as the the vorticity restricted to a plane remains bounded in the scale critical space $L^4_t L^2_x$, where the plane may vary in space and time as long as the gradient of the vector orthogonal to the plane remains bounded. This extends previous work by Chae and Choe that guaranteed that solutions of the Navier--Stokes equation must remain smooth as long as the vorticity restricted to a fixed plane remains bounded in family of scale critical mixed Lebesgue spaces.
This regularity criterion also can be seen as interpolating between Chae and Choe's regularity criterion in terms of two vorticity components and Beir\~{a}o da Veiga and Berselli's regularity criterion in terms of the gradient of vorticity direction.
\end{abstract}

\section{Introduction}

The Navier--Stokes equation is the fundamental equation of fluid mechanics. The incompressible Navier--Stokes equation is given by
\begin{equation} \label{NavierStokes}
\begin{split}
    \partial_t u-\Delta u +(u \cdot \nabla)u+\nabla p=0\\
    \nabla \cdot u=0,
\end{split}
\end{equation}
where $p$ is determined entirely by $u$ by convolution with the Poisson kernel,
\begin{equation}
    p=(-\Delta)^{-1}\sum_{i,j=1}^3
    \frac{\partial u_j}{\partial x_i}
    \frac{\partial u_i}{\partial x_j}.
\end{equation}
The Navier--Stokes equation is best viewed as an evolution equation on the space of divergence free vector fields rather than as a system of equations, and that is the vantage point we will adopt in this paper.

Two other fundamentally important objects for the study of the Navier--Stokes equation are the strain and the vorticity. The strain is the symmetric part of $\nabla u,$ and is given by
\begin{equation}
    S_{ij}=\frac{1}{2}\left(\partial_i u_j
    +\partial_j u_i \right).
\end{equation}
The vorticity is the curl of the velocity, $\omega=\nabla \times u.$ It is a vector representation of the anti-symmetric party of $\nabla u,$ with
\begin{equation}
    A=\frac{1}{2} \left (
\begin{matrix}
0 & \omega_3 & -\omega_2 \\
-\omega_3 & 0 & \omega_1 \\
\omega_2 & -\omega_1 & 0 \\
\end{matrix}
\right ),
\end{equation}
where 
\begin{equation}
    A_{ij}= \frac{1}{2}\left(\partial_i u_j
    -\partial_j u_i \right).
\end{equation}
Note that is implies that for all $v\in\mathbb{R}^3,$
\begin{equation} \label{Av}
    Av=\frac{1}{2}v\times\omega.
\end{equation}
The evolution equation for the strain is given by
\begin{equation} \label{NavierStrain}
\partial_t S -\Delta S +(u \cdot \nabla)S 
+S^2+\frac{1}{4}\omega \otimes \omega - \frac{1}{4}|\omega|^2 I_3+ \Hess(p)=0,
\end{equation}
and the evolution equation for vorticity is given by
\begin{equation} \label{NavierVorticity}
\partial_t \omega- \Delta \omega+ (u\cdot \nabla) \omega- S \omega=0.
\end{equation}

Before we proceed, we should define a number of spaces. We will take $\dot{H}^1\left(\mathbb{R}^3\right)$ to be the homogeneous Hilbert space with norm
\begin{equation}
    \|u\|_{\dot{H}^1}^2=\int_{\mathbb{R}^3}4 \pi^2 |\xi|^2
    \left|\hat{u}(\xi)\right|^2 \diff \xi.
\end{equation}
We will note that for all $u\in\dot{H}^1,$ we have 
\begin{equation}
    \|u\|_{\dot{H}^1}^2=\|\nabla u\|_{L^2}^2.
\end{equation}
We will take the inhomogeneous Hilbert space $H^1\left(\mathbb{R}^3\right)$ to be the space with norm
\begin{equation}
    \|u\|_{H^1}^2=\|u\|_{L^2}^2+\|u\|_{\dot{H}^1}^2
\end{equation}
Finally we will take the mixed Lebesgue space $L^p_t L^q_x$ to be the Banach space
\begin{equation}
    L^p_t L^q_x= 
    L^p\left([0,t);L^q\left(\mathbb{R}^3\right)\right).
\end{equation}
We will note in particular that
\begin{equation}
    \|u\|_{L^\infty_t L^\infty_x}=
    \esssup_{(x,\tau)\in\mathbb{R}^3\times[0,t)}
    |u(x,\tau)|
\end{equation}

In his foundational work on the Navier--Stokes equation, Leray proved the global existence of weak solutions to the Navier--Stokes equation satisfying an energy inequality for arbitrary initial data $u^0\in L^2$ \cite{Leray}. Such solutions, however, are not known to be either smooth or unique.
Kato and Fujita introduced the the notion of mild solutions \cite{KatoFujita}, which are solutions that satisfy the Navier--Stokes equation,
\begin{equation}
    \partial_t u-\Delta u=-(u\cdot\nabla)u- \nabla p,
\end{equation}
in the sense of convolution with the heat kernel as in Duhamel's formula. They used this notion of solution, in particular the higher regularity that can be extracted from the heat kernel, to show that the Navier--Stokes equation has unique, smooth solutions locally in time for arbitrary initial data $u^0\in \dot{H}^1.$ Mild solutions are only known to exist locally in time, however, and so this approach based on the heat semigroup cannot guarantee the existence of global-in-time smooth solutions of the Navier--Stokes equation. It is one of the biggest open questions in the field of nonlinear PDEs whether smooth solutions of the Navier--Stokes equations can develop singularities in finite time. If the $\dot{H}^1$ norm remains bounded this guarantees that a solution of the Navier--Stokes equation must remain smooth, but there is no known bound on the growth of the $\dot{H}^1$ norm of $u.$ One bound we do have for smooth solutions of the Navier--Stokes equation is the energy equality.

\begin{proposition} \label{EnergyEquality}
    Suppose $u \in C\left([0,T_{max});H^1\left(\mathbb{R}^3\right)\right)$ is a mild solution to the Navier--Stokes equation. Then for all $0<t<T_{max}$
    \begin{equation}
        \frac{1}{2}\left\|u(\cdot,t)\right\|_{L^2}^2+
        \int_0^t\|\nabla u(\cdot,\tau)\|_{L^2}^2 \diff\tau
        = \frac{1}{2}\left\|u^0\right\|_{L^2}^2.
    \end{equation}
\end{proposition}

The energy equality gives us bounds on solutions $u$ of the Navier--Stokes equation in $L_t^\infty L^2_x$ and $L^2_t\dot{H}^1_x,$ but this cannot guarantee regularity, because both of these bounds are supercritical with respect to the scale invariance of the Navier--Stokes equation. The solution set of the Navier--Stokes equation is invariant under the rescaling
\begin{equation} \label{NSscaling}
    u^\lambda(x,t)=\lambda u(\lambda x,\lambda^2, t).
\end{equation}
If we do have a bound on $u$ in a scale invariant 
$L^p_t L^q_x$ space then that is enough to guarantee regularity. In particular, the Ladyzhenskaya-Prodi-Serrin regularity criterion \cite{Ladyzhenskaya,Prodi,Serrin} states that if a smooth solution of the Navier--Stokes equation develops a singularity in finite-time $T_{max}<+\infty,$ then for all $\frac{2}{p}+\frac{3}{q}=2, 3<q\leq+\infty,$ then
\begin{equation}
    \int_0^{T_{max}}\|u(\cdot,t)\|_{L^q}^p \diff t
    =+\infty.
\end{equation}
It is easy to check that this regularity criterion is scale critical with respect to the rescaling in \eqref{NSscaling}. This result was extended to the $q=3, p+\infty$ case by Escauriaza, Seregin, and \v{S}ver\'ak \cite{SverakSereginL3}, where they proved that if $T_{max}<+\infty,$ then
\begin{equation}
\limsup_{t \to T_{max}}\|u(\cdot,t)\|_{L^3(\mathbb{R}^3)}=+\infty.
\end{equation}
There have been many extensions of these regularity criteria, far more than we can discuss here, so we will confine ourselves to discussing regularity criteria directly related to this paper.
For a very thorough overview of the literature on regularity criteria for solutions to the Navier-Stokes equation, see Chapter 11 in \cite{NS21}.

One regularity criterion proven by Chae and Choe \cite{ChaeChoe} has a particular geometric significance. Chae and Choe prove a scale critical regularity criterion on two components of vorticity. For a two dimensional solution of the Navier--Stokes equation in the $xy$ plane, the vorticity is entirely in the $z$ direction, perpendicular to the $xy$ plane. This means that if the vorticity restricted to the $xy$ plane or, using rotational invariance, any fixed plane, remains bounded in a scale critical space, then the solution is not too far from being two dimensional. 
The size of the vorticity restricted to a plane can be interpreted as a measure of how fully three dimensional a solution of the Navier--Stokes equation is, so Chae and Choe's regularity criterion can be interpreted as saying that blowup must be fully three dimensional. Chae and Choe's result is the following.

\begin{theorem} \label{Vort2ComP}
Suppose $u\in C\left([0,T_{max});\dot{H}^1\left(\mathbb{R}^3\right)\right)$ is a mild solution of the Navier--Stokes equation, 
and let $v\in\mathbb{R}^3, |v|=1.$
Then for all $\frac{3}{q}+\frac{2}{p}=2,
\frac{3}{2}<q< +\infty,$ there exists $C_q>0$ depending only on $q$ such that for all $0<t<T_{max}$
\begin{equation}
    \|u(t)\|_{\dot{H}^1}^2\leq \|u(0)\|_{\dot{H}^1}^2 \exp\left(C_q\int_0^t
    \left\|v \times \omega(\cdot,\tau)\right
    \|_{L^q}^p\diff\tau\right).
\end{equation}
In particular, if $T_{max}<+\infty,$ then
\begin{equation}
    \int_0^{T_{max}}\left\|v \times \omega
    (\cdot,t)\right\|_{L^q}^p\diff t
    =+\infty.
\end{equation}
\end{theorem}

\begin{remark}
Chae and Choe's result is not stated in this form in \cite{ChaeChoe}. Chae and Choe prove a regularity criterion on two components of the vorticity, 
$(\omega_1,\omega_2,0).$ However, we can see that 
\begin{equation}
    e_3\times \omega= (-\omega_2,\omega_1,0).
\end{equation}
This means that a regularity criterion on $(\omega_1,\omega_2,0)$ is equivalent to a regularity criterion on $e_3\times \omega.$ Using the rotational invariance of the Navier--Stokes equation, this is equivalent to a regularity criterion on $v\times \omega$ for any fixed unit vector $v\in \mathbb{R}^3.$ We note that the Navier--Stokes equation is rationally invariant in the sense that if $Q\in SO(3)$ is a rotation matrix on $\mathbb{R}^3,$ and $u$ is a solution of the Navier--Stokes equation, then $u^Q$ is also a solution of the Navier--Stokes equation where
\begin{equation}
    u^Q(x)=Q^{tr}u\left(Qx \right).
\end{equation}
See chapter 1 in \cite{MajdaBertozzi} for further discussion.
\end{remark}

We will note that because $\omega=\nabla \times u,$ is a derivative of $u,$ the vorticity has the rescaling.
\begin{equation}
    \omega^\lambda(x,t)=
    \lambda^2\omega(\lambda x,\lambda^2 t).
\end{equation}
If $\omega$ is a solution of the vorticity equation, then so is $\omega^\lambda$ for all $\lambda>0.$
Theorem \ref{Vort2ComP} is critical with respect to this rescaling.

Theorem \ref{Vort2ComP} was then extended into Besov spaces by Chen and Zhang \cite{ChenZhangBesov} and more recently into the endpoint Besov space, $\dot{B}^{-\frac{3}{q}}_{\infty,\infty},$ by Guo, Ku\v{c}era, and Skal\'{a}k \cite{GuoBesov}. This is an improvement on Theorem \ref{Vort2ComP}, because we have an embedding $L^q \hookrightarrow \dot{B}^{-\frac{3}{q}}_{\infty,\infty},$ and $\dot{B}^{-\frac{3}{q}}_{\infty,\infty}$ is the largest translation invariant Banach space with the the same scaling relation as $L^q.$ This is therefore, the furthest advance that can be made to Chae and Choe's regularity criterion solely by loosening the assumptions on the space.

Another regularity criterion with geometric significance is the regularity criterion in terms of the positive part of the intermediate eigenvalue of the strain matrix. This was first proven by Neustupa and Penel in \cite{NeuPen1,NeuPen2,NeuPen3} and independently by the author using different methods in \cite{MillerStrain}.

\begin{theorem} \label{StrainRegCrit}
Suppose $u\in C\left([0,T_{max});\dot{H}^1\left(\mathbb{R}^3\right)\right)$ is a mild solution of the Navier--Stokes equation.
Let $\lambda_1(x,t) \leq \lambda_2(x,t) \leq \lambda_3(x,t)$ be the eigenvalues of $S(x,t),$ and 
let $\lambda^+_2(x,t)=\max\left\{0,\lambda_2(x,t)\right\}.$
Then for all $\frac{3}{q}+\frac{2}{p}=2,
\frac{3}{2}<q\leq +\infty,$ there exists $C_q>0$ depending only on $q$ such that for all $0<t<T_{max}$
\begin{equation}
    \|u(t)\|_{\dot{H}^1}^2\leq \|u(0)\|_{\dot{H}^1}^2 \exp\left(C_q\int_0^t
    \left\|\lambda_2^+(\cdot,\tau)\right\|_{L^q}^p\diff\tau\right).
\end{equation}
In particular, if $T_{max}<+\infty,$ then
\begin{equation}
    \int_0^{T_{max}}\left\|\lambda_2^+(\cdot,t)\right\|_{L^q}^p\diff t
    =+\infty.
\end{equation}
\end{theorem}

This regularity criterion gives a geometric characterization of the structure of potential finite-time blowup solutions of the Navier--Stokes equation. Theorem \ref{StrainRegCrit} says that in order for blowup to occur, the flow needs to be stretching in two directions, while compressing more strongly in the third. Points where the strain has two positive eigenvalues, corresponding to stretching in two directions, and one negative eigenvalue, corresponding to compressing more strongly in a third, are the points that drive enstrophy growth and hence blowup. This provides more insight into the qualitative properties of blowup solutions than the regularity criteria that just involve the size of $u$ or $\omega.$
The author also proved the following corollary of Theorem \ref{StrainRegCrit} in \cite{MillerStrain}.

\begin{corollary} \label{StrainAllDirections}
Suppose $u\in C\left([0,T_{max});\dot{H}^1\left(\mathbb{R}^3\right)\right)$ is a mild solution of the Navier--Stokes equation.
Let $v \in L^\infty\left (\mathbb{R}^3 \times
[0,T_{max}];\mathbb{R}^3\right ),$ 
with $|v(x,t)|=1$ almost everywhere.
Then for all $\frac{3}{q}+\frac{2}{p}=2,
\frac{3}{2}<q\leq +\infty,$ there exists $C_q>0$ depending only on $q$ such that for all $0<t<T_{max}$
\begin{equation}
\|u(\cdot,t)\|_{\dot{H}^1}^2 \leq 
\left\|u^0\right\|_{\dot{H}^1}^2
\exp \left (C_q \int_{0}^t\|S(\cdot,\tau)v(\cdot,\tau)
\|_{L^q}^p \diff \tau \right ).
\end{equation}
In particular, if $T_{max}<+\infty,$ then
\begin{equation}
    \int_{0}^{T_{max}}\|S(\cdot,t)v(\cdot,t)\|_{L^q\left (\mathbb{R}^3\right )}^p \diff t
    =+\infty.
\end{equation}
\end{corollary}

This result follows from Theorem \ref{StrainRegCrit} because $\tr(S)=0,$ so $\lambda_2$ is the smallest eigenvalue of $S$ in magnitude. Therefore, if 
$v\in\mathbb{R}^3, |v|=1,$ then
\begin{equation}
    |Sv|\geq |\lambda_2|.
\end{equation}
Corollary \ref{StrainAllDirections} follows immediately from this fact and Theorem \ref{StrainRegCrit}.
We will show in Section 3 that the special case of Corollary \ref{StrainAllDirections} where we take $v\in \mathbb{R}^3$ to be a fixed unit vector is equivalent to Chae and Choe's result, Theorem \ref{Vort2ComP}. 

This raises an interesting question: can Chae and Choe's result be extended to a regularity criterion on $v\times \omega,$ where $v$ is a unit vector that is allowed to vary in space. Clearly, unlike in Corollary \ref{StrainAllDirections}, we won't be able to take an arbitrary unit vector, otherwise we would simply take $v=\frac{\omega}{|\omega|},$ and regularity would be guaranteed for any solution of the Navier--Stokes equation. In fact, using Corollary \ref{StrainAllDirections}, it is possible to improve Chae and Choe's result to one that allows the unit vector $v$ to vary in space and time, so long as $\nabla v$ remains bounded. The main theorem of this paper is the following.

\begin{theorem} \label{MainTheoremIntro}
Suppose $u\in C\left([0,T_{max});H^1\left(\mathbb{R}^3\right)\right)$ is a mild solution of the Navier--Stokes equation.
Suppose $v \in L^\infty\left (\mathbb{R}^3 \times
[0,+\infty);\mathbb{R}^3\right ),$
with $|v(x,t)|=1$ almost everywhere, 
and suppose $\nabla v\in L^\infty_{loc}\left([0,+\infty);
L^\infty\left(\mathbb{R}^3\right)\right),$
Then for all for all $0<t<T_{max}$
\begin{equation}
\|u(\cdot,t)\|_{\dot{H}^1}^2 \leq 
\left\|u^0\right\|_{\dot{H}^1}^2
\exp \left (2916 C_2\left\|u^0\right\|_{L^2}^4\|\nabla v
\|_{L^\infty_t L^\infty_x}^2 +\frac{C_2}{8} 
\int_{0}^t\|v(\cdot,\tau)\times \omega(\cdot,\tau)
\|_{L^2}^4 \diff \tau \right ),
\end{equation}
where $C_2$ is a constant independent of $u,$ taken is in Theorem \ref{StrainRegCrit} and Corollary \ref{StrainAllDirections}.

In particular, if $T_{max}<+\infty,$ then
\begin{equation}
    \int_{0}^{T_{max}}\|v(\cdot,t)\times \omega(\cdot,t)
\|_{L^2}^4 \diff t=+\infty. 
\end{equation}
\end{theorem}

Chae and Choe's result, Theorem \ref{Vort2ComP}, states that for a solution of the Navier--Stokes equation to blowup in finite-time, the vorticity must become unbounded in every fixed plane. For a two dimensional flow in the $xy$ plane, the vorticity is entirely in the $z$ direction, so Theorem \ref{Vort2ComP} can be interpreted as requiring the blowup of vorticity to be globally three dimensional. Theorem \ref{MainTheoremIntro}, strengthens this result, by requiring the vorticity to blowup in every plane, where the plane may vary in space and time so long as the gradient of the vector orthogonal to the plane remains bounded, meaning the geometry of the blowup must be locally three dimensional.

This result highlights the strength of the strain formulation of the Navier--Stokes regularity problem. Using the strain formulation, we are able to get a stronger regularity criterion in terms of the vorticity than by working directly with the vorticity formation. Theorem \ref{MainTheoremIntro} follows as a fairly direct corollary of Corollary \ref{StrainAllDirections} using the relationship between the structure of the strain and vorticity that is imposed by the divergence free condition on the velocity. Because the regularity criterion on $\lambda_2^+$ gives us geometric information that is fundamentally anisotropic in that it does not involve any fixed direction, it allows us to go from component reduction regularity criterion that involve some fixed direction, and therefore guarantee regularity as long as the solution is not too far away from being globally two dimensional in some sense, to a component reduction type regularity criterion that only requires that the solution is not too far away from being locally two dimensional. The regularity criterion in terms of $\lambda_2^+$ first proven by Neustupa and Penel therefore encodes fundamentally anisotropic information about the structure of possible blowup solutions that does not require the imposition of any arbitrary direction. For this reason it is quite powerful.

We should also note that while Chae and Choe's result holds for $\frac{2}{p}+\frac{3}{q}=2, \frac{3}{2}<q<+\infty$, the proof of Theorem \ref{MainTheoremIntro} only holds with $q=2, p=4,$ because the Hilbert space structure of $L^2$ is essential to the proof. It is possible that Theorem \ref{MainTheoremIntro} can be generalized for all $\frac{3}{2}<q<+\infty,$ but this would require much more delicate analysis, and likely somewhat different techniques.

There are several other scale critical, component reduction type regularity criteria. 
For instance, Kukavica and Ziane \cite{Kukavica} showed that if a smooth solution of the Navier--Stokes equation blows up in finite-time $T_{max}<+\infty,$ and if 
$\frac{2}{p}+\frac{3}{q}=2,$ with $\frac{9}{4} \leq q \leq 3,$ then
\begin{equation}
\int_{0}^{T_{max}}\|\partial_3 u(\cdot,t)\|_{L^q}^p 
\diff t= + \infty.
\end{equation}
More recently, it was shown by Chemin and Zhang \cite{CheminZhang} and Chemin, Zhang, and Zhang \cite{CheminZhangZhang} that if $T_{max}<+\infty$ and $4<p<+\infty,$ then
\begin{equation}
\int_0^{T_{max}}\|u_3(\cdot,t)\|_{\dot{H}^{\frac{1}{2}+\frac{2}{p}}}^p \diff t=+\infty.
\end{equation}
Neustupa, Novotn\'{y}, and Penel had previously proven a regularity criterion on $u_3$ \cite{NeustupaOneComp}, which was the first result of this type, but their regularity criteria was not scale critical; it required $u_3$ to be in the subcritical space $L^p_t L^q_x,$ with $\frac{2}{p}+\frac{3}{q}=\frac{1}{2}.$ 
Like Theorem \ref{Vort2ComP} due to Chae and Choe, all of these results say that the flow must be regular as long as it is not too three dimensional in the sense that either $u_3$ or $\partial_3 u$ remains bounded in the appropriate scale critical space. The rotational invariance of the Navier--Stokes equation implies that there is nothing special about the particular direction $e_3,$ so $\partial_3 u$ and $u_3$ can be replaced by $v\cdot \nabla u$ and $u\cdot v$ respectively for any unit vector $v\in\mathbb{R}^3$, but the vector $v$ cannot be allowed to vary in space. These regularity criteria are therefore global anistotropic regularity criteria in that they require solutions which are globally anisotropic in some sense to remain smooth. Theorem \ref{MainTheoremIntro} is significantly stronger because it is a locally anisotropic regularity criterion which only requires solutions to be locally anisotropic to remain smooth. 

There is one recent regularity result that does involve local anisotropy. Kukavica, Rusin and Ziane proved that if $u$ is a suitable weak solution on $\mathbb{R}^3,$ and for some domain $D\subset \mathbb{R}^3\times \mathbb{R}^+,$ we have $\partial_3 u\in L^p_t L^q_x(D),$ with $\frac{2}{p}+\frac{3}{q}=2, \frac{9}{4}\leq q \leq \frac{5}{2},$ then $u$ is H\"older continuous on $D$ \cite{KukavicaLocal}.
This is a locally anisotropic regularity criterion, because we only require control on $\partial_3 u$ in some domain, not globally. This implies that the control on the anisotropy is local. This differs from Theorem \ref{MainTheoremIntro}, which involves an estimate over the whole space, and is therefore a global, not local regularity criterion, but a global regularity criterion that is locally anisotropic, on account of the direction being allowed to vary. We will discuss the relationship between Kukavica, Rusin, and Ziane's regularity criterion and Theorem \ref{MainTheoremIntro} further in section 3.

Another regularity criterion related to Theorem \ref{MainTheoremIntro} is the regularity criterion in terms of the vorticity direction proven by Constantin and Fefferman, which states that the direction of the vorticity must vary rapidly in regions where the vorticity is large in order for a solution of the Navier--Stokes equation to blowup \cite{ConstantinFefferman}.
This was a very important advance in that the rapid change of the vorticity direction in regions of large vorticity has long been at least heuristically understood as a fundamental property of turbulent flow, and Constantin and Fefferman's result shows that any Navier--Stokes blowup solution must be ``turbulent" in this sense.
Indeed, even Leonardo da Vinci's qualitative studies of turbulence in the early 1500s \cite{NatureDaVinci} show an understanding of the fundamental character of a rapid change in the orientation of vortices in turbulent regions of the fluid to the phenomenon of turbulence.

\begin{center}
\begin{figure*}[h]
    \centering
    \includegraphics[width=0.80\textwidth]{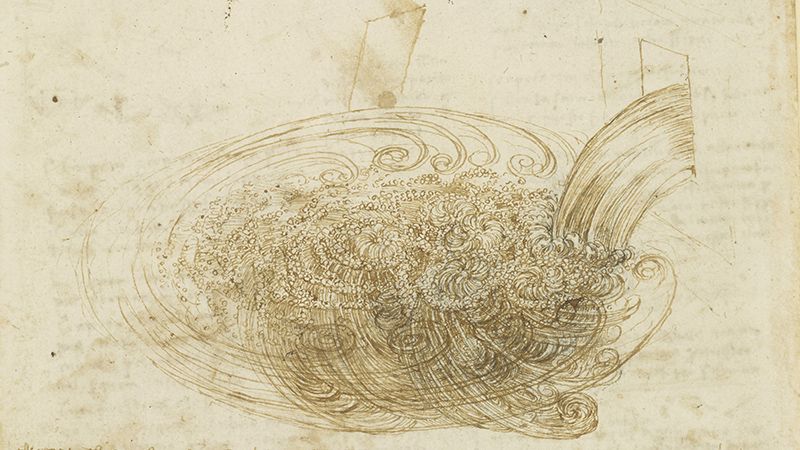}
    \caption{Leonardo da Vinci, \textit{Studies of Turbulent Water}. Royal Collection Trust.}
\end{figure*}
\end{center}

Kolmogorov would rigourously describe the phenomenon depicted purely heuristically by da Vinci with his celebrated theory of turbulent cascades \cite{Kolmogorov}, which remains central to our understanding of turbulence. Kolmogorov showed that turbulence should involve a transfer of energy from lower order frequency modes to higher, with a decay of the energy spectrum 
$E(\xi)\sim |\xi|^{-\frac{5}{3}},$ and also that turbulence is locally isotropic, with no preferred vorticity direction. Constantin and Fefferman were the first to connect the change in the orientation of vortices in turbulent regions to the Navier--Stokes regularity problem, which suggests that blowup for the Navier--Stokes equation is not only of purely mathematical interest, but could have significant physical implications---particularly if there is finite-time blowup---on our phenomenological understanding of turbulence.

Constantin and Fefferman's result was later generalized by Beir\~{a}o da Veiga and Berselli in \cite{daViegaBerselli}. In particular, as a corollary of their refined result that the vorticity direction must vary rapidly in regions with large vorticity, they proved a regularity criterion in terms of the gradient of the vorticity direction, $\nabla \frac{\omega}{|\omega|}.$
\begin{theorem} \label{daVeiga}
Suppose $u\in C\left([0,T_{max});H^1\right)$ is a mild solution of the Navier--Stokes equation that blows up in finite-time $T_{max}<+\infty.$ Then for all $\frac{2}{p}+\frac{3}{q}=\frac{1}{2}, 6\leq q\leq +\infty.$
\begin{equation}
\int_0^{T_{max}}\left\|\nabla \left( \frac{\omega}{|\omega|}
\right)(\cdot,t)\right\|_{L^q}^p \diff t=+\infty.
\end{equation}
\end{theorem}

This is related to Theorem \ref{MainTheoremIntro}, because if we take $v=\frac{\omega}{|\omega|},$ then we find that Theorem \ref{MainTheoremIntro} implies if $T_{max}<+\infty,$ then
\begin{equation}
    \esssup_{0<t<T_{max}}\left\|\nabla\left(\frac{\omega}{|\omega|}
    \right)(\cdot,t)\right\|_{L^\infty}=+\infty.
\end{equation}
If we take the $p=4,q=+\infty$ case of Theorem \ref{daVeiga}, then we can see that if a solution of the Navier--Stokes equation blows up in finite-time $T_{max}<+\infty,$ then
\begin{equation}
    \int_0^{T_{max}}\left\|\nabla \left( \frac{\omega}{|\omega|}
    \right)(\cdot,t)\right\|_{L^\infty}^4 \diff t=+\infty.
\end{equation}

We will note that Beir\~{a}o da Veiga and Berselli's result is stronger than the special case of Theorem \ref{MainTheoremIntro} where we take $v=\frac{\omega}{|\omega|}$ because it not only specifies that $\nabla \frac{\omega}{|\omega|},$ must become unbounded as we approach blowup time, but also a rate of blowup---namely that its integral to the fourth power must go to infinity. We can see that the two extremal cases of Theorem \ref{MainTheoremIntro} are the case where $v$ is constant and hence $\nabla v=0,$ in which case we recover Theorem \ref{Vort2ComP} from Chae and Choe, and the case where $v=\frac{\omega}{|\omega|}$ and hence $v\times \omega=0,$ in which case we recover a weaker form of Theorem \ref{daVeiga}, from Beir\~{a}o da Veiga and Berselli. Therefore this result could be said to interpolate between these two results, although suboptimally at one end. We will discuss why this interpolation is suboptimal in section 3.

In section 2, we will prove the main result of the paper, Theorem \ref{MainTheoremIntro}. In section 3, we will further discuss the relationship between this result and the previous literature.

\section{Proof of the Main Theorem}

We will now prove Theorem \ref{MainTheoremIntro}, which is restated here for the reader's convenience.
\begin{theorem} \label{MainTheorem}
Suppose $u\in C\left([0,T_{max});H^1\left(\mathbb{R}^3\right)\right)$ is a mild solution of the Navier--Stokes equation.
Suppose $v \in L^\infty\left (\mathbb{R}^3 \times
[0,+\infty);\mathbb{R}^3\right ),$
with $|v(x,t)|=1$ almost everywhere, 
and suppose $\nabla v\in L^\infty_{loc}\left([0,+\infty);
L^\infty\left(\mathbb{R}^3\right)\right),$
Then for all for all $0<t<T_{max}$
\begin{equation}
\|u(\cdot,t)\|_{\dot{H}^1}^2 \leq 
\left\|u^0\right\|_{\dot{H}^1}^2
\exp \left (2916 C_2\left\|u^0\right\|_{L^2}^4\|\nabla v
\|_{L^\infty_t L^\infty_x}^2 +\frac{C_2}{8} 
\int_{0}^t\|v(\cdot,\tau)\times \omega(\cdot,\tau)
\|_{L^2}^4 \diff \tau \right ),
\end{equation}
where $C_2$ is taken is in Theorem \ref{StrainRegCrit} and Corollary \ref{StrainAllDirections}.
In particular, if $T_{max}<+\infty,$ then
\begin{equation}
    \int_{0}^{T_{max}}\|v(\cdot,t)\times \omega(\cdot,t)
\|_{L^2}^4 \diff t=+\infty. 
\end{equation}
\end{theorem}

\begin{proof}
We will prove that this theorem is a corollary of Corollary \ref{StrainAllDirections}. In particular, we will show that
\begin{equation}
    \int_0^t \|S(\cdot,\tau)v(\cdot,\tau)\|_{L^2}^4 
    \diff \tau \leq
    2916\|u^0\|_{L^2}^4\|\nabla v\|_{L^\infty_t L^\infty_x}
    + \frac{1}{8}\int_0^t
    \|v(\cdot,\tau)\times
    \omega(\cdot,\tau)\|_{L^2}^4 \diff \tau.
\end{equation}
First we will recall from \eqref{Av} that 
\begin{align}
    \frac{1}{2}v\times \omega&=A v\\
    &=\frac{1}{2}(\nabla u)v-\frac{1}{2}(\nabla u)^{tr} v.
\end{align}
Therefore we find that
\begin{equation} \label{step1}
    \frac{1}{4}\|v\times \omega\|_{L^2}^2=
    \frac{1}{4}\|(\nabla u)v\|_{L^2}^2+
    \frac{1}{4}\|(\nabla u)^{tr}v\|_{L^2}^2
    -\frac{1}{2}\left<(\nabla u)v,(\nabla u)^{tr}v\right>.
\end{equation}
Likewise we may compute that
\begin{equation}
    Sv=\frac{1}{2}(\nabla u)v+\frac{1}{2}(\nabla u)^{tr} v,
\end{equation}
and therefore
\begin{equation}\label{step2}
    \|Sv\|_{L^2}^2=
    \frac{1}{4}\|(\nabla u)v\|_{L^2}^2+
    \frac{1}{4}\|(\nabla u)^{tr}v\|_{L^2}^2
    +\frac{1}{2}\left<(\nabla u)v,(\nabla u)^{tr}v\right>.
\end{equation}
Putting together \eqref{step1} and \eqref{step2}, we find that
\begin{equation}
    \|Sv\|_{L^2}^2=\frac{1}{4}\|v\times \omega\|_{L^2}^2+
    \left<(\nabla u)v,(\nabla u)^{tr}v\right>.
\end{equation}
We know that for all $a,b\geq 0, (a+b)^2\leq 2a^2+2b^2,$
so we find that
\begin{equation}
    \|Sv\|_{L^2}^4=\frac{1}{8}\|v\times \omega\|_{L^2}^4+
    2\left(\left<(\nabla u)v,(\nabla u)^{tr}v\right>\right)^2.
\end{equation}

Now we will need to estimate 
$\left<(\nabla u)v,(\nabla u)^{tr}v\right>.$
By definition we have
\begin{equation}
    \left<(\nabla u)v,(\nabla u)^{tr}v\right>=
    \sum_{i,j,k=1}^3\int_{\mathbb{R}^3} \partial_i u_j v_j \partial_k u_i v_k \diff x.
\end{equation}
Integrating by parts with respect to $x_i$,
and applying the fact that 
$\nabla \cdot u=\sum_{i=1}^3\partial_i u_i=0,$
we find that
\begin{align}
    \left<(\nabla u)v,(\nabla u)^{tr}v\right>&=
    - \sum_{i,j,k=1}^3\int_{\mathbb{R}^3} u_j \partial_k u_i 
    (v_k\partial_i v_j+v_j \partial_i v_k) \diff x\\
    &\leq 
    \sum_{i,j,k=1}^3\int_{\mathbb{R}^3} 2|u|
    |\nabla u| |v| |\nabla v|\\
    &=
    54 \int_{\mathbb{R}^3} |u||\nabla u||\nabla v|\\
    &\leq 
    54 \|u\|_{L^2}\|\nabla u\|_{L^2}\|\nabla v\|_{L^\infty}.
\end{align}
This estimate implies that
\begin{equation}
    \|Sv\|_{L^2}^4\leq \frac{1}{8}\|v\times \omega\|_{L^2}^4
    +5832 \|u\|_{L^2}^2\|\nabla u\|_{L^2}^2
    \|\nabla v\|_{L^\infty}^2.
\end{equation}
Applying the energy equality and our hypothesis on 
$\nabla v,$ we find that
\begin{align}
    \int_0^t \|u(\cdot,\tau)\|_{L^2}^2
    \|\nabla u(\cdot,\tau)\|_{L^2}^2
    \|\nabla v(\cdot,\tau)\|_{L^\infty}^2 \diff\tau 
    &\leq
    \int_0^t \left\|u^0\right\|_{L^2}^2
    \|\nabla u(\cdot,\tau)\|_{L^2}^2
    \|\nabla v\|_{L_t^\infty L^\infty_x}^2 \diff\tau \\
    &=
    \left\|u^0\right\|_{L^2}^2 \|\nabla v\|_{L_t^\infty L^\infty_x}^2
    \int_0^t \|\nabla u(\cdot,\tau)\|_{L^2}^2 \diff\tau \\
    &\leq 
    \frac{1}{2}\left\|u^0\right\|_{L^2}^4 
    \|\nabla v\|_{L_t^\infty L^\infty_x}^2.
\end{align}
Therefore we may conclude that
\begin{equation}
    \int_0^t\|S(\cdot,\tau)v(\cdot,\tau)\|_{L^2}^4 \diff\tau
    \leq 2916 \left\|u^0\right\|_{L^2}^4 
    \|\nabla v\|_{L_t^\infty L^\infty_x}^2+
    \frac{1}{8}\int_0^t\|v(\cdot,\tau)\times 
    \omega(\cdot,\tau)\|_{L^2}^4 \diff\tau
\end{equation}
Applying Corollary \ref{StrainAllDirections}, this completes the proof.
\end{proof}

\section{Relationship to the previous literature}

In this section, we will further discuss the relationship between the results in this paper and the previous literature. In particular, we will show that Theorem \ref{StrainAllDirections}, the regularity criterion on $Sv,$ is equivalent to Chae and Choe's regularity criterion on $v\times \omega$ in the special case where we take $v\in\mathbb{R}^3$ to be a fixed unit vector. In order to do this we will first need to introduce the Helmholtz decomposition of vector fields in $L^q, 1<q<+\infty$ into gradients and divergence free vector fields.

\begin{proposition} \label{Helmholtz}
Suppose $1<q<+\infty.$ For all $v\in 
L^q(\mathbb{R}^3;\mathbb{R}^3)$ there exists a unique $u\in L^q(\mathbb{R}^3;\mathbb{R}^3),$ $\nabla \cdot u=0$ and $\nabla f \in L^q(\mathbb{R}^3;\mathbb{R}^3)$ such that $v=u+\nabla f.$ Note because we do not have any assumptions of higher regularity, we will say that $\nabla \cdot u=0,$ if for all $\phi \in C_c^\infty(\mathbb{R}^3)$
\begin{equation}
    \int_{\mathbb{R}^3}u \cdot \nabla \phi=0,
\end{equation}
and we will say that $\nabla f$ is a gradient if for all $w\in C_c^\infty(\mathbb{R}^3;\mathbb{R}^3), \nabla \cdot w=0,$ we have
\begin{equation}
    \int_{\mathbb{R}^3}\nabla f \cdot w=0.
\end{equation}
Furthermore there exists $C_q\geq 1$ depending only on $q,$ such that 
\begin{equation}
    \|u\|_{L^q}\leq C_q \|v\|_{L^q},
\end{equation}
and
\begin{equation}
    \|\nabla f\|_{L^q} \leq C_q \|v\|_{L^q}.
\end{equation}

Using this decomposition we will define the projection onto the space of divergence free vector fields $P_{df}: L^q \to L^q$ to be
\begin{equation}
    P_{df}(v)=u,
\end{equation}
and the projection onto the space of gradients $P_{gr}: L^q \to L^q$ to be
\begin{equation}
    P_{gr}(v)=\nabla f.
\end{equation}
\end{proposition}

We can now show that for any fixed unit vector $v\in\mathbb{R}^3$
$\|Sv\|_{L^q}$ and $\|v\times\omega\|_{L^q}$ are equivalent norms.

\begin{proposition} \label{NormEquiv}
Suppose $\nabla \cdot u=0,$ and $\nabla u\in L^q$ Then for all unit vectors $v\in\mathbb{R}^3,$ and for all $1<q<+\infty$
\begin{align}
    \|(v\cdot \nabla) u\|_{L^q}&\leq C_q \|v\times \omega\|_{L^q},\\
    \|(v\cdot \nabla) u\|_{L^q}&\leq C_q \|2 S v\|_{L^q},\\
    \|\nabla (u\cdot v)\|_{L^q}&\leq C_q \|v\times \omega\|_{L^q},\\
    \|\nabla (u\cdot v)\|_{L^q}&\leq C_q \|2 S v\|_{L^q},
\end{align}
where $C_q$ is taken as in Proposition \ref{Helmholtz}. 
Furthermore,
\begin{equation}
    \frac{1}{2C_q}\|v\times \omega\|_{L^q}\leq\|2Sv\|_{L^q}
    \leq 2C_q \|v\times \omega\|_{L^q}.
\end{equation}
\end{proposition}

\begin{proof}
First we will observe that the rotational invariance of the space of divergence free vector fields means that we can take $v=e_3$ without loss of generality. It is a simple calculation to see that
\begin{equation}
    e_3\times \omega= \partial_3 u-\nabla u_3.
\end{equation}
Likewise we can see that 
\begin{equation}
    2S e_3= \partial_3 u+\nabla u_3.
\end{equation}
Clearly $\nabla u_3$ is a gradient, and we can also see that $\partial_3 u$ is divergence free because 
\begin{align}
    \nabla \cdot \partial_3 u&=
    \partial_3 \nabla \cdot u\\
    &=0.
\end{align}
This implies that
\begin{align}
    \partial_3 u&= P_{df}(e_3\times \omega),\\
    \partial_3 u&= P_{df}(2S e_3).
\end{align}
Applying Proposition \ref{Helmholtz} we can therefore conclude that 
\begin{align}
    \|\partial_3 u\|_{L^q}&\leq C_q \|e_3\times \omega\|_{L^q},\\
    \|\partial_3 u\|_{L^q}&\leq C_q\|2S e_3\|_{L^q}.
\end{align}
Likewise we can observe that
\begin{align}
    \nabla u_3&= -P_{gr}(e_3\times \omega),\\
    \nabla u_3&= P_{gr}(2S e_3),
\end{align}
and apply Proposition \ref{Helmholtz} to find that
\begin{align}
    \|\nabla u_3\|_{L^q}&\leq C_q \|e_3\times \omega\|_{L^q},\\
    \|\nabla u_3\|_{L^q}&\leq C_q\|2S e_3\|_{L^q}.
\end{align}
We will now finish the proof by applying the triangle inequality and concluding that
\begin{align}
    \|e_3 \times\omega \|_{L^q}&=
    \|\partial_3 u-\nabla u_3\|_{L^q}\\
    &\leq \|\partial_3 u\|_{L^q}+\|\nabla u_3\|_{L^q}\\
    &\leq 2C_q \|2Se_3\|_{L^q},
\end{align}
and that
\begin{align}
    \|2Se_3\|_{L^q}&=
    \|\partial_3 u+\nabla u_3\|_{L^q}\\
    &\leq \|\partial_3 u\|_{L^q}+\|\nabla u_3\|_{L^q}\\
    &\leq 2C_q \|e_3 \times\omega \|_{L^q}.
\end{align}
This completes the proof.
\end{proof}

Proposition \ref{NormEquiv} implies that in the range $\frac{9}{4}\leq q\leq 3,$ Kukavica and Ziane's regularity criterion on $\partial_3 u$ implies Chae and Choe's regularity criterion on $e_3 \times \omega$ because
\begin{equation}
    \|\partial_3 u\|_{L^q}\leq 
    C_q \|e_3 \times \omega\|_{L^q}.
\end{equation}
Likewise, Chemin, Zhang, and Zhang's regularity criterion on $u_3$ implies Chae and Choe's regularity criterion on $e_3\times \omega$ when $4<p<+\infty.$
We can see this by applying the fractional Sobolev embedding 
$ W^{1,q}\left (\mathbb{R}^3 \right ) \hookrightarrow
\dot{H}^{\frac{1}{2}+\frac{2}{p}}\left (\mathbb{R}^3 \right )$
when $\frac{2}{p}+\frac{3}{q}=2,$ and Proposition \ref{NormEquiv} to find
\begin{align}
\|u_3\|_{\dot{H}^{\frac{1}{2}+\frac{2}{p}}}
&\leq 
C_q\|\nabla u_3\|_{L^q}\\
&\leq
\Tilde{C}_q \|e_3\times \omega\|_{L^q}
\end{align}

We will also note that Theorem \ref{MainTheoremIntro} differs in a fundamental way from the locally anisotropic regularity criterion proven by Kukavica, Rusin, and Ziane in \cite{KukavicaLocal}. This is true in particular because while 
\begin{equation}
    \|\partial_3 u\|_{L^q\left(\mathbb{R}^3\right)} \leq C_q
    \|e_3 \times \omega\|_{L^q\left(\mathbb{R}^3\right)},
\end{equation}
no such inequality holds if we are not on the whole space. For arbitrary $\nabla u\in L^q, \nabla \cdot u=0,$ and arbitrary domains $\Omega\subset \mathbb{R}^3,$ there is no inequality of the form
\begin{equation}
    \|\partial_3 u\|_{L^q\left(\Omega\right)} \leq C_q
    \|e_3 \times \omega\|_{L^q\left(\Omega\right)}.
\end{equation}
The proof of Proposition \ref{NormEquiv} relies on the $L^q$ boundedness of the Helmholtz decomoposition, which in turn relies on the $L^q$ boundedness of the Riesz transform. This is something which only holds globally, not locally, because the Riesz transform is a nonlocal operator. $e_3 \times \omega$ controls $\partial_3 u$ globally in $L^q,$ but not locally in $L^q.$ This means that while Kukavica and Ziane's regularity criterion for $\partial_3 u \in L^p_t L^q_x\left(\mathbb{R}^3\right)$ implies Chae and Choe's regularity criterion for $e_3 \times \omega \in L^p_t L^q_x\left(\mathbb{R}^3\right)$ for $\frac{9}{4}\leq q \leq 3,$ the local anisotropic regularity criterion with $\partial_3 u\in L^p_t L^q_x\left(D\right)$ and the locally anisotropic regularity criterion on $v\times \omega,$ where $v$ is allowed to vary do not have this same relationship.

Finally, we will note that when we set $v=\frac{\omega}{|\omega|},$ Theorem \ref{MainTheoremIntro} requires that if our solution of the Navier--Stokes equation blows up in finite-time $T_{max}<+\infty,$ then
\begin{equation}
    \esssup_{(x,t)\in\mathbb{R}^3\times [0,T_{max})}
    \left|\nabla \left(\frac{\omega}{|\omega|}
    \right)(x,t)\right|=+\infty.
\end{equation}
However, as we discussed in the introduction, Beir\~{a}o da Veiga and Berselli proved the stronger result that under these conditions
\begin{equation}
    \int_0^{T_{max}}\left\|\nabla\left(
    \frac{\omega}{|\omega|}\right)(\cdot,t)
    \right\|_{L^\infty}^4 \diff t=+\infty.
\end{equation}
For this reason, it seems like it ought to be possible to relax the condition in Theorem \ref{MainTheoremIntro} from $\nabla v\in L^\infty_t L^\infty_x$  to $\nabla v\in L^4_t L^\infty_x.$ The methods used to prove Theorem \ref{MainTheoremIntro} would at first glance suggest this would be possible, but the difficulty is that when integrating by parts we cannot get all of the derivatives off of $u$ and onto $v,$ which is what leads to the sub-optimal bound.

\begin{remark}
Suppose for all $u\in H^1, \nabla \cdot u=0, v\in L^\infty, |v(x)|=1$ almost everywhere $x\in\mathbb{R}^3$ we had a bound of the form
\begin{align}
    |\left<(\nabla u) v, (\nabla u)^{tr}v\right>|
    &\leq \label{DesiredIneq}
    C\int_{\mathbb{R}^3} |u|^2|\nabla v|^2\\
    &\leq
    C \|u\|_{L^2}^2\|\nabla v\|_{L^\infty}^2.
\end{align}
Then using the energy inequality we could conclude that
\begin{align}
    \int_0^t\left|\left<(\nabla u) v, 
    (\nabla u)^{tr}v\right>(\tau)\right|^2\diff \tau
    &\leq
    C^2\int_0^t\|u(\cdot,\tau)\|_{L^2}^4
    \|\nabla v(\cdot,\tau)\|_{L^\infty}^4 \diff \tau\\
    &\leq 
    C^2 \left\|u^0\right\|_{L^2}^4\int_0^t
    \|\nabla v(\cdot,\tau)\|_{L^\infty}^4\diff \tau,
\end{align}
and we could relax the requirement in Theorem \ref{MainTheoremIntro} to $\nabla v\in L^4_t L^\infty_x$ using precisely the same proof.
The difficulty is that, as much as at first glance it would appear that an inequality of the form \eqref{DesiredIneq} should hold using integration by parts, after multiple attempts to do so by the author it does not appear possible to push all the derivatives off of $u$ and onto $v$ in the manner desired. The asymmetry from the transpose, $(\nabla u)^{tr},$ does not seem to allow this. It is definitely possible, perhaps even likely, that the condition can be relaxed to $\nabla v\in L^4_t L^\infty_x,$ so that the interpolation between Chae and Choe's regularity criterion and Beir\~{a}o da Veiga and Berselli's will be optimal, but it would require a fundamentally different proof, because it will not follow as a corollary the regularity criterion on $\lambda_2^+$ in Theorem \ref{StrainRegCrit}.
\end{remark}

\bibliographystyle{plain}
\bibliography{Bib}
\end{document}